\documentclass[twoside, a4paper, 10pt]{amsart}
\title[ ]{Twisted recurrence via polynomial walks}
\usepackage{amsaddr}
\author{Kamil Bulinski and Alexander Fish}
\address{School of Mathematics and Statistics F07, University of Sydney, NSW 2006, Australia}
\email{K.Bulinski@maths.usyd.edu.au}
\email{alexander.fish@sydney.edu.au}
\usepackage{amsfonts}
\usepackage{amsthm}
\usepackage{verbatim}
\usepackage{amsmath, amssymb}
\usepackage{tikz}
\usetikzlibrary{matrix, arrows}
\usepackage{listings}
\usepackage{color}
\usepackage{listings}
\usepackage[all]{xy}
\usepackage[pdftex,colorlinks,linkcolor=blue,citecolor=blue]{hyperref}
\usepackage{graphicx}
\usepackage{float}
\usepackage[margin=3cm]{geometry}
\usepackage{bigints}
\usepackage{dsfont}
\begin{document}
\maketitle
\raggedbottom

\newcommand{\cA}{\mathcal{A}}
\newcommand{\cB}{\mathcal{B}}
\newcommand{\cC}{\mathcal{C}}
\newcommand{\cD}{\mathcal{D}}
\newcommand{\cE}{\mathcal{E}}
\newcommand{\cF}{\mathcal{F}}
\newcommand{\cG}{\mathcal{G}}
\newcommand{\cH}{\mathcal{H}}
\newcommand{\cI}{\mathcal{I}}
\newcommand{\cJ}{\mathcal{J}}
\newcommand{\cK}{\mathcal{K}}
\newcommand{\cL}{\mathcal{L}}
\newcommand{\cM}{\mathcal{M}}
\newcommand{\cN}{\mathcal{N}}
\newcommand{\cO}{\mathcal{O}}
\newcommand{\cP}{\mathcal{P}}
\newcommand{\cQ}{\mathcal{Q}}
\newcommand{\cR}{\mathcal{R}}
\newcommand{\cS}{\mathcal{S}}
\newcommand{\cT}{\mathcal{T}}
\newcommand{\cU}{\mathcal{U}}
\newcommand{\cV}{\mathcal{V}}
\newcommand{\cW}{\mathcal{W}}
\newcommand{\cX}{\mathcal{X}}
\newcommand{\cY}{\mathcal{Y}}
\newcommand{\cZ}{\mathcal{Z}}
\newcommand{\bA}{\mathbb{A}}
\newcommand{\bB}{\mathbb{B}}
\newcommand{\bC}{\mathbb{C}}
\newcommand{\bD}{\mathbb{D}}
\newcommand{\bE}{\mathbb{E}}
\newcommand{\bF}{\mathbb{F}}
\newcommand{\bG}{\mathbb{G}}
\newcommand{\bH}{\mathbb{H}}
\newcommand{\bI}{\mathbb{I}}
\newcommand{\bJ}{\mathbb{J}}
\newcommand{\bK}{\mathbb{K}}
\newcommand{\bL}{\mathbb{L}}
\newcommand{\bM}{\mathbb{M}}
\newcommand{\bN}{\mathbb{N}}
\newcommand{\bO}{\mathbb{O}}
\newcommand{\bP}{\mathbb{P}}
\newcommand{\bQ}{\mathbb{Q}}
\newcommand{\bR}{\mathbb{R}}
\newcommand{\bS}{\mathbb{S}}
\newcommand{\bT}{\mathbb{T}}
\newcommand{\bU}{\mathbb{U}}
\newcommand{\bV}{\mathbb{V}}
\newcommand{\bW}{\mathbb{W}}
\newcommand{\bX}{\mathbb{X}}
\newcommand{\bY}{\mathbb{Y}}
\newcommand{\bZ}{\mathbb{Z}}

\newcounter{dummy} \numberwithin{dummy}{section}

\theoremstyle{definition}
\newtheorem{mydef}[dummy]{Definition}
\newtheorem{prop}[dummy]{Proposition}
\newtheorem{corol}[dummy]{Corollary}
\newtheorem{thm}[dummy]{Theorem}
\newtheorem{lemma}[dummy]{Lemma}
\newtheorem{eg}[dummy]{Example}
\newtheorem{notation}[dummy]{Notation}
\newtheorem{remark}[dummy]{Remark}
\newtheorem{claim}[dummy]{Claim}
\newtheorem{Exercise}[dummy]{Exercise}
\newtheorem{question}[dummy]{Question}

\newtheorem*{thm*}{Theorem}

\begin{abstract} In this paper we show how polynomial walks can be used to establish a twisted recurrence for sets of positive density in $\bZ^d$. In particular, we prove that if  $\Gamma \leq GL_d(\bZ)$ is finitely generated by unipotents and acts irreducibly on $\bR^d$, then for any set $B \subset \bZ^d$ of positive density,  there exists $k \geq 1$ such that for any $v \in k \bZ^d$ one can find $\gamma \in \Gamma$ with $\gamma v \in B - B$. Our method does not require the linearity of the action, and we prove a twisted recurrence for semigroups of maps from $\bZ^d$ to $\bZ^d$ satisfying some irreducibility and polynomial assumptions. As one of the consequences, we prove a non-linear analog of Bogolubov's theorem -- for any set $B \subset \bZ^2$ of positive density, and $p(n) \in \bZ[n]$,  $p(0) = 0$, $\deg{p} \geq 2$ there exists $k \geq 1$ such that $k \bZ \subset \{ x - p(y) \, | \, (x,y) \in B-B \}$. 
Unlike the previous works on twisted recurrence that used recent results of Benoist-Quint and Bourgain-Furman-Lindenstrauss-Mozes on equidistribution of random walks on automorphism groups of tori, our method relies on the classical Weyl equidistribution for polynomial orbits on tori.
 \end{abstract} 
\section{Introduction}

\subsection{Background.} 
One of the first discovered instances of recurrence is Furstenberg-Sark\"ozy theorem \cite{Fur77}. \cite{Sa}. It says that  if a set $B \subset \bZ$ has positive upper Banach density, i.e. 
\[
d^*(B) = \limsup_{b-a \to \infty} \frac{| B \cap [a,b)|}{b-a} > 0,
\]
then for any polynomial $p(n) \in \bZ[n]$, with $p(0) = 0$, there exists $n \geq 1$ such that $p(n) \in B-B$. Furstenberg related this statement to a claim on recurrence within measure-preserving systems. Namely, for any $\bZ$-measure preserving system $(X,\mu,T)$ and any measurable set $B \subset X$ with $\mu(B) > 0$ there exists $n \geq 1$ such that 
\[
\mu(B \cap T^{p(n)} B) > 0.
\]
Later on, Furstenberg-Katznelson-Weiss \cite{FKW} showed by the ergodic method that for any measurable set $B \subset \bR^d$, $d \geq 2$, of positive upper Banach density, i.e.,
\[
d^*(B) = \limsup_{R \to \infty} \sup_{x \in \bR^d} \frac{|B \cap (x +  Q_R)|}{|Q_R|} > 0,
\]
where $Q_R = [0,R) \times [0,R) \ldots \times [0,R) \subset \bR^d$, and the $| \cdot |$ denotes the Lebesgue measure,
there exists $R_0 > 0$ such that the set of distances in $B$, $$\Delta_B = \{ \| b_1 - b_2 \|^2 \,\,  | \,\,  b_1,b_2 \in B \},$$ contains $[R_0,\infty)$.
\subsection{Earlier work on twisted patterns.} We recall that for a set $B \subset \bZ^d$, the upper Banach density of $B$ is defined by
\[
d^*(B) = \limsup_{R \to \infty} \sup_{x \in \bZ^d} \frac{|B \cap (x +  Q_R)|}{R^d},
\]
where $Q_R = [0,R) \times [0,R) \ldots \times [0,R) \subset \bR^d$, and the $| \cdot |$ denotes the cardinality of finite sets. Bohr sets are important examples of the positive upper Banach density sets. Recall, that if $\tau$ is a homomorphism from $\bZ^d$ to a compact group $K$ with dense image, and $U \subset K$ is open, then the set $B = \tau^{-1}(U)$ is called a Bohr set. If, in addition, $K$ is connected, then $B$ is an aperiodic Bohr set.
The following result of Magyar on the distances in positive density subsets of $\bZ^d$  may be viewed as a discrete analogue of the Furstenberg-Katznelson-Weiss theorem.

\begin{thm}[Magyar \cite{Magyar}] Let $B \subset \bZ^d$, where $d \geq 5$,  be a set of positive upper Banach density.
Then there exists a positive integer $q$ such that $$q\bZ_{\geq 0} \subset \left \{ \| b_1-b_2 \|^2 \text{ }| \text{ } b_1,b_2 \in B \right \}.$$

\end{thm}

A recent series of works, initiated by Bj\"orklund and the second author \cite{BFChar} and further developed by Bj\"orklund and the first author \cite{BBTwisted}, demonstrates that analogous results hold if one replaces the squared distance $\| \cdot \|^2$ with other quadratic forms or certain homogeneous polynomials. The techniques in those works were Ergodic Theoretic and exploited the fact that such functions were preserved by a sufficiently large and algebraically structured subgroup of $\operatorname{SL}_d(\bZ)$, to which one could apply recent measure rigidity and equidistribution results of Benoist-Quint \cite{BQ1} \cite{BQ3} and those of Bourgain-Furman-Lindenstrauss-Moses \cite{BFLM}. The general statement obtained in \cite{BBTwisted} may be formulated as follows.

\begin{thm}[Twisted patterns] \label{thm: twisted patterns intro} Let $\Gamma \leq \operatorname{SL}_d(\bZ)$ be a non-trivial finitely generated subgroup such that the linear action $\Gamma \curvearrowright \bR^d$ is irreducible and the Zariski closure $G \leq \operatorname{SL}_d(\bR)$ of $\Gamma$ is Zarisk-connected, a semi-simple Lie group and has no compact factors. Then for all sets $B \subset \bZ^d$ of positive upper Banach density and positive integers $m$, there exists a positive integer $k=k(B,m)$ such that the following holds: For all $a_1, \ldots, a_m \in k \bZ^d$ there exists $\gamma_1, \ldots, \gamma_m \in \Gamma$ and $b \in B$ such that $$\{\gamma_1 a_1, \ldots, \gamma_m a_m \} \subset B-b.$$ 

In particular, if $F: \bZ^d \to \mathcal{S}$ is a function ($\mathcal{S}$ is any set) preserved by $\Gamma \leq \operatorname{SL}_d(\bZ)$ (this means that $F \circ \gamma=F$ for all $\gamma \in \Gamma$) then for all $R_1, \ldots R_m \in F(k\bZ^d)$ there exists $b \in B$ such that $$ R_1, \ldots, R_m \in F(B-b). $$
Moreover, if $B$ is an aperiodic Bohr set, then  one may take $k = 1$.

\end{thm}
\medskip

Note that the $m=1$ case was obtained by Bj\"orklund and the second author in \cite{BFChar} and is an analogue of Magyar's theorem (with $F$ playing the role of $\| \cdot \|^2$). For $m \geq 1$, this result is an analogue of certain \textit{pinned distance} results (cf. \cite{LyallMagyar})

We now explore some examples observed in the aforementioned papers (we will mainly focus on the $m=1$ case).

\begin{eg}[Non-positive-definite quadratic forms, see \cite{BBTwisted}] \label{eg: non-def quad forms intro} Fix integers $p,q \geq 1$, with $d:=p+q \geq 3$ and let $Q:\bZ^{d} \to \bZ$ by an integral quadratic form of signature $p,q$ (for example, $Q:\bZ^3 \to \bZ$ given by $Q(x,y,z)=x^2+y^2-z^2$). Then the integral special orthogonal $$\Gamma=\operatorname{SO}(Q)(\bZ) := \{ \gamma \in \operatorname{SL}_d(\bZ) \text{ } | \text{ } Q \circ \gamma=\gamma \}$$ of $Q$ satisfies the hypothesis of Theorem~\ref{thm: twisted patterns intro}. Hence we obtain the following analogue of Magyar's theorem for $Q$: For all $B \subset \bZ^d$, of positive upper Banach density, there exists an integer $k\geq 1$ such that $$Q(B-B) \supset Q(k\bZ^d) = k^2 Q(\bZ^d).$$   \end{eg}

\begin{eg}[Characteristic polynomials of traceless integer matrices, see \cite{BFChar}] \label{eg: char poly intro} Let $n \geq 2$ and let $\mathfrak{sl}_n(\bZ)$ denote the integer matrices of trace zero, which is isomorphic (as an abelian group) to $\bZ^{n^2-1}$ and hence has a notion of upper Banach density (we fix an identification of $\mathfrak{sl}_n(\bZ)$ with $\bZ^{n^2-1}$ for the remainder of this example) . It turns out that the \textit{adjoint representation} $$\operatorname{Ad}: \operatorname{SL}_n (\bZ) \curvearrowright \mathfrak{sl}_n(\bZ) \cong \bZ^{n^2-1},$$ given by $\operatorname{Ad}(g)A=gAg^{-1}$ for $g \in \operatorname{SL}_n (\bZ)$ and $A \in  \mathfrak{sl}_n(\bZ)$ has the desired properties; namely, $\Gamma=\operatorname{Ad}(\operatorname{SL}_n(\bZ)) \leq \operatorname{SL}_d(\bZ)$ (where $d=n^2-1$) satisfies the hypothesis of Theorem~\ref{thm: twisted patterns intro}. Since the adjoint representation preserves characteristic polynomials, we may consider the characteristic polynomial map $$\operatorname{char}: \mathfrak{sl}_n(\bZ) \to \bZ[t]$$ given by $\operatorname{char}(A)=\operatorname{det}(A-t \operatorname{I})$ and conclude that for all $B \subset \mathfrak{sl}_n(\bZ)$ of positive upper Banach density, we have that $\operatorname{char}(B-B)$ contains $$\operatorname{char}(k \text{ } \mathfrak{sl}_n(\bZ)) = \left \{ \sum_{j=0}^n k^{n-j} a_j t^j \text{ }| \text{ } a_0, \ldots a_n \in \bZ \text{ with } a_n=1 \text{ and }a_{n-1} =0. \right \} \subset  \bZ[t]$$ for some integer $k\geq 1$. Likewise, the determinant map $\operatorname{det}: \mathfrak{sl}_n(\bZ) \to \bZ$ is preserved by the adjoint representation and hence for all $B \subset \mathfrak{sl}_n(\bZ)$ of positive upper Banach density, we have that $\operatorname{det}(B-B)$ contains a non-trivial subgroup $k\bZ$, for some integer $k \geq 1$. \end{eg}

\subsection{A self-contained approach} As alluded to above, the proof of Theorem~\ref{thm: twisted patterns intro} given in \cite{BFChar} and \cite{BBTwisted} relies on very deep results\footnote{Alternatively, one can use the work of Bourgain-Furman-Lindenstrauss-Mozes \cite{BFLM}, see e.g. \cite{Fish}} of Benoist-Quint obtained in \cite{BQ1} and \cite{BQ3}. One of the goals of this paper is to develop a much more elementary approach (by avoiding these works of Benoist-Quint) to such extensions of Magyar's theorem, which will also allow us to furnish some new examples that are not obtainable from the previous works \cite{BFChar} and \cite{BBTwisted}. We begin with one such general result that we are able to obtain by completely self-contained and classical means.

\begin{thm} \label{thm: unipotent generators intro}  Let $\Gamma \leq \operatorname{GL}_d(\bZ)$, where $d \geq 2$, be a subgroup such that 

\begin{enumerate}
	\item The linear action of $\Gamma$ on $\bR^d$ is irreducible.
	\item There exists a finite set $S \subset \Gamma$ of unipotent matrices which generate $\Gamma$.

\end{enumerate}

Then for all $B \subset \bZ^d$ of positive upper Banach density and integers $m \geq 1$, there exists a positive integer $k=k(B,m)$ such that the following holds: For all $a_1, \ldots, a_m \in k\bZ^d$ there exists $\gamma_1, \ldots, \gamma_m \in \Gamma$  and $b \in B$ such that $\{\gamma_1a_1, \ldots, \gamma_m a_m\} \in B-b$.

\end{thm}

\begin{remark}
In this and all subsequent theorems, in the case of $B$ being an aperiodic Bohr set,  one may take $k = 1$.
\end{remark}

In particular, we recover Example~\ref{eg: char poly intro} as well as some (but not all) of the cases in Example~\ref{eg: non-def quad forms intro}. To see this, note that $\operatorname{SL}_n(\bZ)$ is generated by a finite set of unipotents (for example, elementary matrices) and if $g \in \operatorname{SL}_n(\bZ)$ is unipotent, then so is the endomorphism $\operatorname{Ad}(g): A \mapsto gAg^{-1}$ (for irreducibility, see Appendix~\ref{appendix: irreducibility zariski dense}). On the other hand, there are no non-trivial unipotent matrices in the group $\operatorname{SO}(Q)(\bZ)$, where $Q(x,y,z)=x^2+y^2-3z^2$. To see this, note that $Q(v) \neq 0$ for all non-zero $v \in \bQ^3$, and hence, by Proposition 5.3.4 in \cite{MorrisArithmetic}, $\operatorname{SO}(Q)(\bZ)$ is a cocompact lattice in $\operatorname{SO}(Q)(\bR)$, from which it follows, by Corollary 4.4.4 in \cite{MorrisArithmetic}, that $\operatorname{SO}(Q)(\bZ)$ contains no non-trivial unipotent elements. Let us mention however that we are still able to recover Example~\ref{eg: non-def quad forms intro} for many interesting quadratic forms.

\begin{eg} The group $\operatorname{SO}(Q)(\bZ)$, where $Q(x,y,z)=xy-z^2$ contains a subgroup $\Gamma$ that acts irreducibly on $\bR^3$ and is generated by unipotents. This is clear once we write $$Q(x,y,z)=\operatorname{det} \begin{pmatrix}
  z & -y \\
  x & -z
 \end{pmatrix}$$ and regard $Q$ as the determinant map on $\mathfrak{sl}_2(\bZ)$. So we may take $\Gamma=\operatorname{Ad}(\operatorname{SL}_2(\bZ))$, as observed in Example~\ref{eg: char poly intro}, which is generated by unipotents.

\end{eg}

\begin{eg}\label{eg: x^2 +y^2 - z^2 intro} Let $Q(x,y,z)=x^2-y^2-z^2$ and observe that $$Q(x,y,z)= \operatorname{det} \begin{pmatrix} z & -(x+y) \\ x-y & -z \end{pmatrix}.$$ Hence we may regard $Q$ as the determinant map on the abelian subgroup $$\Lambda = \left \{ \begin{pmatrix} a_{11} & a_{12} \\ a_{21} & a_{22} \end{pmatrix} \in \mathfrak{sl}_2(\bZ) \text{ }| \text{ } a_{21} \equiv a_{12} \mod 2 \right\}.$$ Notice however that $\Gamma=\operatorname{Ad}(\Gamma_0)$ preserves this abelian subgroup, where $$\Gamma_0 = \left \langle \begin{pmatrix}1 & 2 \\ 0 & 1 \end{pmatrix}, \begin{pmatrix}1 & 0 \\ 2 & 1 \end{pmatrix} \right \rangle,$$ and acts irreducibly on $\mathfrak{sl}_2(\bR)$ (see Appendix~\ref{appendix: irreducibility zariski dense}). Hence Theorem~\ref{thm: unipotent generators intro} applies. \end{eg}

In fact, we may extend this example to higher dimensions as follows.

\begin{prop} \label{prop: higher dimension} Consider the quadratic form $Q(x_1, \ldots, x_p,y_1, \ldots y_q)=x_1^2 + \cdots + x_p^2 - y_1^2 - \cdots - y_q^2$ where $p \geq 1$ and $q \geq 2$. Then $SO(Q)(\bZ)$ contains a subgroup acting irreducibly on $\bR^d$ that is generated by finitely many unipotent  elements. In particular, we recover the following non-definite analogue of Magyar's theorem: If $B \subset \bZ^{p+q}$ has positive upper Banach density, then $Q(B-B)$ contains a non-trivial subgroup of $\bZ$.

\end{prop}

Again, we refer the reader to Appendix~\ref{appendix: irreducibility zariski dense} for the details.

\subsection{A generalization to non-linear actions} 

It turns out that our approach to Theorem~\ref{thm: unipotent generators intro} may be generalized to non-linear semigroups. In particular, an analogous result holds for semigroups consisting of certain \textit{polynomial}, rather than linear, transformations on $\bZ^d$. This, in turn, may be used to obtain a Magyar-type result for certain non-homogeneous polynomials $F: \bZ^d \to \bZ$. Before stating the main result, we will firstly illustrate this phenomena with the following example.

\begin{thm} \label{thm: xy-P(z) is Magyar} Let $P(z) \in \bZ[z]$ be a polynomial of degree $\geq 2$, with $P(0)=0$, and let $F(x,y,z)=xy-P(z)$. Then for all $B \subset \bZ^3$ of positive upper Banach density there exists a positive integer $k=k(B)$ such that $k\bZ \subset F(B-B)$. More generally, for all positive integers $m$ there exists a positive integer $q=q(B,m)$ such that for all $R_1, \ldots, R_m \in q\bZ$ there exists $b \in B$ such that $$ \{ R_1, \ldots, R_M \} \subset F(B-b).$$ 

\end{thm}

In particular, we recover Example~\ref{eg: non-def quad forms intro} for all quadratic forms of the form $Q(x,y,z)=xy-dz^2$. We remark that the the group of linear transformations which preserve $xy-z^3$ spectacularly fails the hypothesis of Theorem~\ref{thm: twisted patterns intro}, as it preserves the $z=0$ plane (hence the linear action is not irreducible) and is actually a finite group. We now turn to describing our general result. To do this, we begin by generalizing the notion of unipotency to non-linear polynomial transformations.

\begin{mydef} If $\Gamma \subset \{\bZ^d \to \bZ^d\}$ is a semigroup of mappings, then a \textit{polynomial walk in $\Gamma$} is a map $s:\bZ_{\geq 0} \to \Gamma$ such that

\begin{enumerate} 
	\item There exist polynomials $$q_1(t_1, \ldots, t_{d+1}), \ldots, q_d(t_1, \ldots, t_{d+1}) \in \bZ[t_1, \ldots, t_{d+1}]$$ such that for each $x=(x_1, \ldots, x_d) \in \bR^d$ and non-negative integer $n$ we have that $$s(n)x=(q_1(n,x_1, \ldots, x_d), \ldots, q_d(n, x_1, \ldots, x_d)).$$ 
	\item $s(0)=\operatorname{Id}_{\bR^d}$.
\end{enumerate} \end{mydef}

To see the connection to unipotency, observe that if $\gamma \in \operatorname{SL}_d(\bZ)$ is unipotent then the sequence $S(n)=\gamma^n$ is a polynomial walk in $\operatorname{SL}_d(\bZ)$. 

We now turn to extending the notion of an irreducible representation to non-linear actions. 

\begin{mydef} We say that a set $A \subset \bR^d$ is \textit{hyperplane-fleeing} if for all proper affine subspaces\footnote{By an affine subspace, we mean a translated linear subspace.} $W \subset \bR^d$ we have that $A \not \subset W$. A sequence $a_1,a_2, \ldots$ is defined to be \textit{hyperplane-fleeing} if the set $\{a_1, a_2, \ldots \}$ is hyperplane-fleeing.

\end{mydef}

For instance, a linear representation of a group $\Gamma$ on a $\bR^d$, where $d \geq 2$, is irreducible if and only if the $\Gamma$-orbit of each non-zero vector is hyperplane-fleeing. To see this, note that if $w \in \bR^d$ is a non-zero vector such that $\Gamma w \subset W$ for some proper affine subspace $W$, then, as $d \geq 2$, we can find $\gamma \in \Gamma$ such that $\gamma w-w \neq 0$ (otherwise $\bR w$ is a one dimensional representation). But then $\Gamma (\gamma w -w) \subset W-W$, is contained in a proper linear subspace. We are now in a position to state our main result, which is a non-linear generalization of Theorem~\ref{thm: unipotent generators intro}.

\begin{thm}\label{thm: non-linear twisted patterns intro} Let $\Gamma \subset \{ \bZ^d \to \bZ^d \}$ be a semigroup of maps and suppose that there exists a finite set $\{s_1, \ldots, s_r \}$ of polynomial walks in $\Gamma$ such that $ \{ s_j(n) \text{ } | \text{ } n \geq 0 \text{ and } j=1, \ldots, r \}$ generates $\Gamma$. Then for $B \subset \bZ^d$ of positive upper Banach density and $m \in \bZ_{>0}$, there exists a positive integer $k=k(B,m)$ such that the following holds: For all $v_1, \ldots, v_m \in k\bZ^d$ with each $\Gamma v_i$ hyperplane-fleeing, there exists $\gamma_1, \ldots, \gamma_m \in \Gamma$ such that $$ \gamma_1v_1, \ldots, \gamma_m v_m \in B-b \quad \text{for some } b \in B.$$

\end{thm}
\medskip

By use of Theorem~\ref{thm: non-linear twisted patterns intro} we obtain the following non-linear extension of the classical Bogolubov's theorem \cite{Bog}.

\begin{thm}\label{Bogolubov} Let $F(x,y)=x-P(y)$ where $P(y) \in \bZ[y]$ is of degree $\geq 2$ with $P(0)=0$. Then for all $B \subset \bZ^2$ of positive upper Banach density, we have that $F(B-B)$ contains a non-trivial subgroup of $\bZ$.

\end{thm}

\subsection{Ergodic formulation} As in the previous works \cite{BFChar} and \cite{BBTwisted}, the correspondence principle of Furstenberg \cite{Fur77} allows us to reduce Theorem~\ref{thm: non-linear twisted patterns intro} to the following Ergodic-theoretic \textit{twisted multiple recurrence} statement.

\begin{thm}\label{thm: non-linear twisted-multiple recurrence intro} Let $\Gamma \subset \{ \bZ^d \to \bZ^d \}$ be a semigroup of maps and suppose that there exists a finite set $\{s_1, \ldots, s_r \}$ of polynomial walks in $\Gamma$ such that $ \{ s_j(n) \text{ } | \text{ } n \geq 0 \text{ and } j=1, \ldots, r \}$ generates $\Gamma$. Suppose that $T: \bZ^d \curvearrowright (X,\mu)$ is a measure preserving system. Then for all $B \subset X$ with $\mu(B)>0$ and positive integers $m$ there exists a positive integer $k=k(B,m,\epsilon)>0$ such that the following holds: If $v_1, \ldots, v_m \in k\bZ^d$ are such that each orbit $\Gamma v_i$ is hyperplane-fleeing, then there exist $\gamma_1, \ldots, \gamma_m \in \Gamma$ such that $$\mu(B \cap T^{-\gamma_1 v_1}B \cap \cdots \cap T^{-\gamma_m v_m}B) > \mu(B)^{m+1}-\epsilon.$$

\end{thm}
\medskip

\textit{Acknowledgment.} A part of this work has been carried out while the second author was a Feinberg visiting scholar at  Weizmann Institute of Science, Israel. He would like to thank Omri Sarig for the hospitality and encouragement. The authors would like to thank Michael Bj\"orklund for encouraging and stimulating discussions.

\section{Constructing hyperplane-fleeing polynomial walks}

We now turn to the first part of the strategy of our proof, which is a construction of hyperplane-fleeing polynomial walks in hyperplane-fleeing orbits. We first state our result in the special case where our action is linear.

\begin{thm} \label{thm: hyerplane-fleeing walks for linear actions} Let $\Gamma \leq \operatorname{GL}_d(\bZ)$, where $d \geq 2$, be a subgroup such that 

\begin{enumerate}
	\item The linear action of $\Gamma$ on $\bR^d$ is irreducible.
	\item There exists a finite set $S \subset \Gamma$ of unipotent matrices which generate $\Gamma$.

\end{enumerate}

Then for all $v \in \bZ^d \setminus \{0\}$ there exists a sequence $\gamma_1, \gamma_2, \ldots \in \Gamma$  such that $$\gamma_n v=(p_1(n), \ldots, p_d(n))$$ for some polynomials $p_1(t), \ldots, p_n(t) \in \bZ[t]$ such that no non-trivial linear combination of these polynomials is constant (i.e., $1,p_1(t), \ldots, p_d(t)$ are linearly independent over $\bZ$). In other words, the sequence $\gamma_n v$ is hyperplane-fleeing.

\end{thm}

Let us now return to the more abstract setting of non-linear actions. We begin with a simple lemma which will allow us to perform certain useful operations on polynomial walks.

\begin{lemma} \label{lemma: operations on polynomial walks} If $S,R: \bZ_{\geq 0} \to \Gamma$ are polynomial walks then:

\begin{enumerate} 
	
	\item The map $S\circ R:\bZ_{\geq 0} \to \Gamma$ given by $(S\circ R)(n)=S(n) \circ R(n)$ is a polynomial walk.
	\item If $\ell$ is a positive integer then $S(n^{\ell})$ is a polynomial walk.
	\item \label{item: scaling polynomial walk} For all integers $k$ and $n \geq 0$, we have that $S(k n)(k\bZ^d) \subset k\bZ^d$.

\end{enumerate}

\begin{proof}[Proof of (\ref{item: scaling polynomial walk})] By definition $S(0)0=\operatorname{Id}0=0$. This means that each entry of $S(k n)(k x_1, \ldots, k x_d)$ is an integer polynomial, with zero constant term, in $kn,kx_1, \ldots, kx_d$ and hence a multiple of $k$ whenever $n,x_1, \ldots, x_d$ are integers.  \end{proof}

\end{lemma}

The non-linear generalization of Theorem~\ref{thm: hyerplane-fleeing walks for linear actions} may now be formulated as follows.

\begin{thm} \label{thm: hyperplane fleeing polynomial walks} Let $\Gamma \subset \{ \bZ^d \to \bZ^d \}$ be a semigroup of maps and suppose that there exists a finite set $\{s_1, \ldots, s_r \}$ of polynomial walks in $\Gamma$ such that $ \{ s_j(n) \text{ } | \text{ } n \geq 0 \text{ and } j=1, \ldots, r \}$ generates $\Gamma$. Suppose that $v \in \bR^d$ satisfies the property that the orbit $\Gamma v$ is not contained in any proper affine subspace of $\bR^d$. Then there exists a polynomial walk $S:\bZ_{\geq 0} \to \Gamma$ such that $$S(n) v=(p_1(n), \ldots, p_d(n))$$ for some polynomials $p_1(t), \ldots, p_n(t) \in \bZ[t]$ such that no non-trivial $\bR$-linear combination of these polynomials is constant (i.e., $1,p_1(t), \ldots, p_d(t)$ are linearly independent over $\bR$). In other words, the sequence $S(n)(v)$ is \textit{hyperplane-fleeing}. 

\end{thm}

\begin{proof} Use cyclic notation to define $s_N = s_{N \text{ } \mathrm{mod} \text{ }r}$ for all positive integer $N$ and let $$A_N=\{ s_N(t_N) \cdots s_1(t_1) v \text{ } | \text{ } t_1, \ldots, t_N \in \mathbb{Z}_{\geq 0} \}.$$ Notice that $$\Gamma v = \bigcup_{N \geq 1} A_N$$ (this follows from $s_j(0)=\operatorname{Id}$) and $A_N \subset A_{N+1}$. Let $$V_N = \{ L:\bR^d \to \bR \text{ }| \text{ }L \text{ is affine such that } L(a)=0 \text{ for all } a \in A_N \}$$ denote the vector space of all affine maps (linear plus a constant) which annihilate $A_N$. Then $V_N \supset V_{N+1}$ and $\bigcap_{N \geq 1} V_N = \{0 \}$ is trivial since $\Gamma v$ is hyperplane-fleeing. So there exists a positive integer $N$ such that $V_N = \{0 \}$, which means that $A_N$ is hyperplane-fleeing. Hence the vector of polynomials $$(p_1(t_1, \ldots, t_N), \ldots, p_d(t_1, \ldots, t_N)) = s_N(t_N) \cdots s_1(t_1) v \in (\bZ[t_1, \ldots t_N])^d$$ satisfies the property that no non-trivial linear combination of its entries is constant. It is easy to construct a rapidly growing sequence $e_1<e_2< \ldots <e_N$ of positive integers such that the polynomials $$p_1(t^{e_1}, \ldots, t^{e_N}), \ldots, p_d(t^{e_1}, \ldots, t^{e_N}) \in \bZ[t]$$ also satisfy the property that no non-trivial linear combination of them is constant (one may take $e_k=R^k$ where $R$ is the largest power of some $t_i$ appearing in the polynomial vector above). This means that $$S(n)=s_N(n^{e_N}) \cdots s_1(n^{e_1})$$ satisfies the desired conclusion (it is a polynomial walk by Lemma~\ref{lemma: operations on polynomial walks}).  \end{proof}

\section{A uniform approximation on Ergodic averages along polynomial walks}

In this section, we study ergodic averages along hyperplane-fleeing polynomial walks. Our main technical result is a uniform estimate of their limits, which depends only on the rational spectrum.


\begin{mydef} We let $$\textbf{Rat}= \{ \chi \in \widehat{\bZ^d} \text{ }| \text{ The image of } \chi \text{ is finite.}  \}$$ denote the set of rational characters, where $\widehat{\bZ^d}$ denotes the group of characters on $\bZ^d$ (homomorphisms from $\bZ^d$ to the multiplicative group of unit complex numbers). If $T:\bZ^d \curvearrowright (X,\mu)$ is a measure preserving system and $\chi \in \widehat{\bZ^{d}}$ is a character then a $\chi$-eigenfunction of $(X,\mu,T)$ is a function $f \in L^2(X,\mu)$ such that $T^v f=\chi(v)f$ for all $v \in \bZ^d$. Furthermore, if $\chi \in \textbf{Rat}$ then we say that $f$ is a rational eigenfunction. We let $L^2_{\textbf{Rat}}(X,\mu,T)$ denote the rational Kronecker factor, i.e., the norm closed subspace spanned by rational eigenfunctions. We let $P_{\textbf{Rat}}: L^2 (X,\mu,T) \to L^2_{\textbf{Rat}}(X,\mu,T)$ denote the orthogonal projection onto the rational Kronecker factor.  \end{mydef}

\begin{prop}  Suppose that $T: \bZ^d \curvearrowright (X,\mu)$ is a measure preserving system. Let $f \in L^2(X,\mu)$ be orthogonal to the rational Kronecker factor of $(X,\mu,T)$. Then for all polynomials $p_1(n), \ldots p_d(n) \in \bZ[n]$ such that no non-trivial $\bR$-linear combination of them is constant we have that $$\lim_{N \to \infty}\left \| \frac{1}{N} \sum_{n=1}^N T^{p(n)} f \right \|_2 =0,$$ where $p(n)=(p_1(n), \ldots, p_d(n)).$ 

\end{prop}

\begin{proof} By the spectral theorem, there exists a positive Borel measure $\sigma$ on $\bT^d$ such that $$\langle T^v f, f \rangle = \int_{\bT^d} e(\langle v, \theta \rangle ) d\sigma(\theta) \quad \text{for all }v \in \bZ^d. $$  Since $f \in L^2_{\textbf{Rat}}(X,\mu,T)^{\perp}$ we have that $\sigma ( \bQ^d/\bZ^d)=0$ and hence $$\left \|\frac{1}{N} \sum_{n=1}^N T^{p(n)} f \right \|^2_2 =\int_{\Omega} \left| \frac{1}{N} \sum_{n=1}^N e( \langle  p(n), \theta\rangle ) \right|^2 d\sigma(\theta)  $$ where $\Omega=\bT^d \setminus (\bQ^d /\bZ^d).$ But by Weyl's polynomial equidistribution theorem, we have that $$ \lim_{N\to \infty} \frac{1}{N} \sum_{n=1}^N e( \langle  p(n), \theta\rangle ) =0 \quad \text{for all } \theta \in \Omega$$ since $\langle p(n), \theta \rangle \notin \bR + \bQ[n]$, i.e., it is a polynomial in $n$ with at least one irrational non-constant term. The dominated convergence theorem now completes the proof. \end{proof}

\begin{prop} \label{prop: polynomial average approximates P_rat}   Suppose that $T: \bZ^d \curvearrowright (X,\mu)$ is a measure preserving system. Then for $f \in L^2(X,\mu)$ and $\epsilon>0$ there exists a positive integer $k=k(f,\epsilon)$ such that the following holds: For all polynomials $p_1(n), \ldots, p_d(n) \in \bZ[n]$ such that no non-trivial $\bR$-linear combination of them is constant, we have that the limit $$Q_{p} f := \lim_{N\to \infty} \frac{1}{N} \sum_{n=1}^N T^{p(n)} f  $$ exists in the $L^2(X,\mu)$ norm topology, where $p(n)=(p_1(n), \ldots, p_d(n))$, and if $p(\bZ) \subset k\bZ^d$ then $$\left \| Q_{p} f- P_{\textbf{Rat}}f \right \|_2 < \epsilon.$$

\end{prop}

\begin{proof} Decompose $f=h+h^{\perp}$ where $h$ is in the rational Kronecker factor and $h^{\perp}$ is orthogonal to the rational Kronecker factor. We already know that the limit $Q_p h^{\perp}$ exists and equals $0$. Now write $$h= \sum_{\chi \in \textbf{Rat}} c_{\chi} h_{\chi}$$ where $h_{\chi}$ is a $\chi$ eigenfunction (meaning that $T^v f=\chi(v) f$ for all $v \in \bZ^d$) with $\|h_{\chi}\|_2=1$. Now if $\chi$ is rational then there exists a positive integer $k_{\chi}$ such that $\chi( v )=1$ for all $v \in k_{\chi} \bZ^d$, which means that the sequence $T^{p(n)}h_{\chi}$ is $k_{\chi}$ periodic and so we have the existence of the limit $$Q_p h_{\chi} :=\lim_{N\to \infty} \frac{1}{N} \sum_{n=1}^N T^{p(n)} h_{\chi} = \frac{1}{k_{\chi}} \sum_{n=1}^{k_{\chi}} \chi(p(n)) h_{\chi}.$$ The existence of the limit $Q_p f$ now immediately follows from a basic approximation argument. From this expression it also follows that if $p(\bZ) \subset k_{\chi}\bZ^d$ then $$Q_p h_{\chi}=h_{\chi}.$$ This means that if $k$ is a positive integer such that $p(\bZ) \subset k\bZ^d$, we will have $$ \| Q_p f - h\|_2 = \| Q_p h - h\|_2 \leq 2 \sqrt{ \underset{ \{ \chi \text{ } : \text{ } k \not \equiv 0 \text{ }\mathrm{mod}\text{ } k_{\chi} \} } \sum c^2_{\chi} }.  $$ But we may choose a highly divisible enough positive integer $k$ such that this upper bound is less than $\epsilon$. Notice that $k$ depends only on $\epsilon$ and the coefficients $c_{\chi}$ and is uniform across all such polynomial sequences $p(n)$.

\end{proof}

\section{Proof of the main recurrence result}

We are now in a position to prove Theorem~\ref{thm: non-linear twisted-multiple recurrence intro}, which we restate verbatim for the convenience of the reader as follows.

\begin{thm*} Let $\Gamma \subset \{ \bZ^d \to \bZ^d \}$ be a semigroup of maps and suppose that there exists a finite set $\{s_1, \ldots, s_r \}$ of polynomial walks in $\Gamma$ such that $ \{ s_j(n) \text{ } | \text{ } n \geq 0 \text{ and } j=1, \ldots, r \}$ generates $\Gamma$. Suppose that $T: \bZ^d \curvearrowright (X,\mu)$ is a measure preserving system. Then for all $B \subset X$ with $\mu(B)>0$ and positive integers $m$ there exists a positive integer $k=k(B,m,\epsilon)>0$ such that the following holds: If $v_1, \ldots, v_m \in k\bZ^d$ satisfy the property that each orbit $\Gamma v_i$ is hyperplane-fleeing then there exist $\gamma_1, \ldots, \gamma_m \in \Gamma$ such that $$\mu(B \cap T^{-\gamma_1 v_1}B \cap \cdots \cap T^{-\gamma_m v_m}B) > \mu(B)^{m+1}-\epsilon.$$

\end{thm*}

\begin{proof} Apply Proposition~\ref{prop: polynomial average approximates P_rat} with $f=\mathds{1}_B$ and let $k=k(f, \frac{\epsilon}{m})>0$ be as in the corresponding conclusion of that proposition. We apply Theorem~\ref{thm: hyperplane fleeing polynomial walks} to each $v_i$ to obtain polynomial walks $R_i: \bZ_{\geq 0} \to \Gamma$ such that each sequence $R_i(n)v_i$ is hyperplane-fleeing. Note the crucial observation that $v_i \in k\bZ^d$ implies that the sequence $P_i(n) := R_i(k n)v_i$ satisfies $P_i(n) \in k \bZ^d$ for all $n \in \bZ_{\geq 0}$ (See Lemma~\ref{lemma: operations on polynomial walks}) and is also hyperplane-fleeing (if a polynomial sequence enters a hyperplane infinitely many times, it is always there). Hence we may indeed apply Proposition~\ref{prop: polynomial average approximates P_rat} with $p(n)=P_i(n)$ to get that \begin{align}\label{proof estimate: P_rat close to Q_{P_i}} \left \|\lim_{N\to \infty} \frac{1}{N} \sum_{n=1}^N T^{P_i(n)}\mathds{1}_B - P_{\textbf{Rat}}\mathds{1}_B \right \|_2 <\frac{\epsilon}{m}. \end{align} 

Now consider the following average of correlations \begin{align*} C(N_1, \ldots, N_m) &= \frac{1}{N_1 \cdots N_m} \sum_{n_1=1}^{N_1} \cdots \sum_{n_m=1}^{N_m} \mu(B \cap T^{-P_1(n_1)}B \cap \cdots \cap T^{-P_m(n_m)}B) \\ &=  \frac{1}{N_1 \cdots N_m} \sum_{n_1=1}^{N_1} \cdots \sum_{n_m=1}^{N_m} \int_X \mathds{1}_B \left(T^{P_1(n_1)} \mathds{1}_B \right) \cdots \left( T^{P_m(n_m)}\mathds{1}_B \right) d\mu  \end{align*} 

and notice that $$\lim_{N_1 \to \infty} \cdots \lim_{N_m \to \infty} C(N_1, \ldots, N_m) = \int_X \mathds{1}_B (Q_{P_1} \mathds{1}_B) \cdots  (Q_{P_m} \mathds{1}_B) d\mu $$  where $P_i(n)=R_i(n)v_i$ and $Q_{P_i}$ is as defined in the statement of Proposition~\ref{prop: polynomial average approximates P_rat}. But now applying the estimates (\ref{proof estimate: P_rat close to Q_{P_i}}) iteratively for $i=1, \ldots, m$ together with the fact that $\| P_{\textbf{Rat}}\mathds{1}_{B} \|_{\infty} \leq 1$ (this follows from the well known fact $P_{\textbf{Rat}}$ is a conditional expectation) we get that $$ \int_X \mathds{1}_B (Q_{P_1} \mathds{1}_B) \cdots  (Q_{P_m} \mathds{1}_B) d\mu > \int_X \mathds{1}_B (P_{\textbf{Rat}}\mathds{1}_B)^m d\mu - \epsilon.$$ 

Now we use the fact that $(P_{\textbf{Rat}}\mathds{1}_B)^m \in L^2_{\textbf{Rat}}(X,\mu,T)$ (since $L^2_{\textbf{Rat}}(X,\mu,T) \cap L^{\infty}(X,\mu)$ is an algebra) and Jensen's inequality to get  \begin{align*}\int_X \mathds{1}_B (P_{\textbf{Rat}}\mathds{1}_B)^m d\mu & = \langle \mathds{1}_B ,(P_{\textbf{Rat}}\mathds{1}_B)^m \rangle \\  & \geq  \langle P_{\textbf{Rat}}\mathds{1}_B ,P_{\textbf{Rat}} (P_{\textbf{Rat}}\mathds{1}_B)^m \rangle \\ & =  \langle P_{\textbf{Rat}}\mathds{1}_B , (P_{\textbf{Rat}}\mathds{1}_B)^m \rangle \\ & \geq \left( \int P_{\textbf{rat}}\mathds{1}_B d\mu \right)^{m+1} \\ &= \left( \int \mathds{1}_B d\mu \right)^{m+1}.    \end{align*} This finishes the proof, as we have shown that $$\lim_{N_1 \to \infty} \cdots \lim_{N_m \to \infty} C(N_1, \ldots, N_m) > \mu(B)^{m+1} - \epsilon. $$     \end{proof}

\section{Proofs of Theorems \ref{thm: xy-P(z) is Magyar} and \ref{Bogolubov}}

\textit{Proof of Theorem ~\ref{thm: xy-P(z) is Magyar}.} Let $P(z) \in \bZ[z]$ be a polynomial of degree $\geq 2$ with $P(0)=0$. The polynomial $F(x,y,z)=xy-P(z)$ has the following symmetries which are polynomial walks:

$$ S_1(n) (x,y,z)=(x,y + H(n,x,z), z+nx) $$ and $$S_2(n) (x,y,z) = (x+H(n,y,z),y ,z+ny)$$ where $$H(n,x,z) = \frac{P(z+nx)-P(z)}{x} \in \bZ[n,x,z]$$ is a polynomial. Observe that, as a polynomial in $n$ (with coefficients in $\bZ[x,z]$) the degree of $H(n,x,z)$ is the degree of $P$ (at least $2$), with leading term of the form $Cx^{\deg P-1}n^{\deg P}$ for some integer $C \neq 0$. Let $\Gamma$ be the semigroup generated by these polynomial walks.

\textbf{Claim:} For all $v_0=(x_0,y_0,z_0) \in \bZ^3$ with $x_0 \neq 0$, we have that $\Gamma v$ is hyperplane-fleeing.

We may assume, without loss of generality, that $y_0 \neq 0$ since we may find $n \in \bN$ such that $v_0'=S_1(n)v_0=(x_0, y'_0, z'_0)$ where $y'_0 \neq 0$ (since, as observed above, $H(n,x_0,z_0)$ is non-constant when $x_0 \neq 0$). Suppose firstly that the sequence $S_1(n)(v_0)$ is contained in a hyperplane $ax+by+cz=d$. As observed earlier, $H(n,x_0,z_0)$ has leading term $Cx_0^{\deg P-1}n^{\deg P}$ and hence since $x_0 \neq 0$, it has degree $\deg P \geq 2$. From this we conclude that $b=0$. But since $z_0+nx_0$ is non-constant, we conclude that $c=0$. Thus we have shown that if $x_0 \neq 0$, we have that $S_1(n)(x_0,y_0,z_0)$ can only be contained in the hyperplane $x=x_0$. Finally, by symmetry, we have that $S_2(n)(v_0)$ can only be contained in the hyperplane $y=y_0$ (in particular, not $x=x_0$). This means that $\Gamma v_0$ is hyperplane-fleeing. 

Using this claim and Theorem~\ref{thm: non-linear twisted patterns intro}, we reduce Theorem~\ref{thm: xy-P(z) is Magyar} to the statement that for all integers $k\geq 1$, the set $F\left(k \bZ^3 \setminus \left( \{0 \} \times \bZ^2 \right ) \right)$ contains $q\bZ$ for some integer $q \geq 1$. But $F(k,ka,0)=k^2a$ and so we may take $q=k^2$.

\qed

\textit{Proof of Theorem~\ref{Bogolubov}.} Let $F(x,y)=x-P(y)$ where $P(y) \in \bZ[y]$ has degree at least $2$ with $P(0)=0$. Notice that $F$ is preserved by the polynomial walk $S(n)(x,y)=(x+P(y+n) -  P(y), y+n)$ for $n \in \bZ_{\geq 0}$. Moreover, for all $v_0=(x_0,y_0) \in \bZ^2$, we have that $S(n)v_0$ is hyperplane-fleeing since $x_0+P(y_0+n)-P(y_0)$ is non-linear in $n$. Consequently, Theorem~\ref{thm: non-linear twisted patterns intro} is applicable to the group generated by this polynomial walk. 

\qed

\appendix

\section{Basic algebra} \label{appendix: irreducibility zariski dense}

\begin{lemma} Let $\Gamma \leq G \leq \operatorname{GL}_n(\bR)$ be groups such that $G$ is the Zariski closure of $\Gamma$. Suppose that $\rho:G \to \operatorname{GL}_d(\bR)$ is an irreducible representation such that $\rho$ is a polynomial map. Then the restriction $\rho |_{\Gamma}: \Gamma \to \operatorname{GL}_d(\bR)$ is also irreducible. 
\end{lemma}

\begin{proof} Suppose on the contrary that the restriction is reducible. This means that there exists a proper linear subspace $W \leq \bR^d$ and $w \in W$ such that $\rho(\Gamma) w \subset W$. Let $\pi:\bR^d \to \bR^d/W$ denote the quotient map. Then $P:G \to \operatorname{GL}_d(\bR)$ given by $P(g)=\pi (\rho(g) w)$ is a polynomial in $g$ which vanishes for all $g \in \Gamma$. Since $G$ is the Zariski closure of $\Gamma$, we get that $P$ also vanishes on $G$ and hence $\rho(G)w \subset W$, which contradicts the irreducibility of $\rho$.  \end{proof}

\begin{lemma} Let $a,b \in \bZ \setminus \{0\}$ be non-negative integers. Then the subgroup $$\Gamma_0 = \left \langle \begin{pmatrix}1 & a \\ 0 & 1 \end{pmatrix}, \begin{pmatrix}1 & 0 \\ b & 1 \end{pmatrix} \right \rangle$$ is Zariski dense in $\operatorname{SL}_2(\bR)$. 

\end{lemma}

\begin{proof} Let $$U(t) = \begin{pmatrix}1 & t \\ 0 & 1 \end{pmatrix}.$$ We wish to show that the Zariski closure of $\Gamma_0$ contains $U(t)$ and its transpose, for all $t \in \bR$, as these generate $\operatorname{SL}_2(\bR)$. Now suppose that $P:\operatorname{SL}_2(\bR) \to \bR$ is a polynomial map which vanishes on all of $\Gamma_0$. Then, in particular, the polynomial $R: \bR \to \bR$ given by $R(x)=P(U(x))$ vanishes on the infinite set $a\bZ$, and so $R(x)$ is the zero polynomial. Hence $P$ vanishes on $U(t)$, for all $t \in \bR$. This shows that $U(t)$ is in the Zariski closure, and a similair argument applies to its transpose.  \end{proof}

\begin{eg} The adjoint representation $\operatorname{Ad}:\operatorname{SL}_d(\bR) \to \operatorname{GL}(\mathfrak{sl}_d(\bR))$ is a polynomial map. It is an irreducible representation and hence the above lemmata may be applied to verify the claims in Example~\ref{eg: x^2 +y^2 - z^2 intro}. \end{eg}

\begin{lemma}\label{Gluing irreducible reps} Suppose that $V = W_0 \oplus W \oplus W_1$ is a direct sum of linear spaces, with $W \neq \{0 \}$, and let $V_0 =W_0+W$ and $V_1=W+W_1$. Let $\Gamma_1, \Gamma_2 \leq \operatorname{GL}(V)$ be groups such that for $i=0,1$ we have that 

\begin{enumerate} 
	\item $\Gamma_iV_i=V_i$ and the representation $\Gamma_i \curvearrowright V_i$ is irreducible.
	\item $\Gamma_i$ acts trivially on $W_{1-i}$ (i.e., $\gamma v=v$ for all $\gamma \in \Gamma_i$ and $v \in W_{1-i}$). \end{enumerate}
\end{lemma}

Then the action of $\Gamma=\langle \Gamma_0 \cup \Gamma_1 \rangle$ on $V$ is irreducible.

\begin{proof} It will be more convenient to work with the group algebras $A_i=\bF[\Gamma_i]$ and $A= \bF[\Gamma]$. Let $v \in V$ be non-zero. We wish to show that $Av = V_0+V_1$. Since $v \neq 0$ we get that either $A_1v$ or $A_2v$ contains a vector with $W$ component non-zero (from the irreduciblity of $\Gamma_1$ and $\Gamma_2$). Let $v' \in Av$ be such a vector and write $v'=w'_0+w'+w'_1$ where $w' \in W$, $w'_0 \in W_0$ and $w'_1 \in W_1$. As $w' \neq 0$, the hypothesis on $\Gamma_i$ gives us that $A_iw'=V_i$ but $A_i'w_{1-i}=\{w_{1-i}\}$ for $i=0,1$. This means that $A_0v'=V_0+w'_1$ and $A_1v'=w'_0+V_1$. Hence $Av$ contains  $A_0v'+A_1v' = V_0+V_1+w'_0+w'_1=V_0+V_1=V$.

\end{proof}

\begin{prop} Consider the quadratic form $Q(x_1, \ldots, x_p,y_1, \ldots y_q)=x_1^2 + \cdots + x_p^2 - y_1^2 - \cdots - y_q^2$ where $p \geq 1$ and $q \geq 2$. Then $SO(Q)(\bZ)$ contains a subgroup acting irreducibly on $\bR^d$ that is generated by finitely many unipotent  elements.

\end{prop}

\begin{proof} Let $e_1, \ldots, e_p, u_1, \ldots,u_q$ denote the standard basis for $\bR^{p+q}$. For $1 \leq a \leq p$ and $1 \leq b <c \leq q$ let $V_{a,b,c}$ be the vector space spanned by $e_a,u_b,u_c$. Then we have a group $\Gamma_{a,b,c} \leq \operatorname{SL}_d(\bZ)$, generated by a finite set of unipotents, which acts irreducibly on $V_{a,b,c}$, preserves $x_a^2-y_a^2-y_b^2$ and acts trivially on the complement (w.r.t. this basis) of $V_{a,b,c}$. By iteratively applying Lemma~\ref{Gluing irreducible reps} to suitable combinations of the $\Gamma_{a,b,c}$, we may construct the desired subgroup. 

\end{proof}

\end{document}